\documentclass[12pt,reqno]{amsart}
\usepackage{amsmath,amssymb,amscd,verbatim,color}

\usepackage[backref]{hyperref}

\usepackage{amsfonts}
\usepackage[top=35mm, bottom=35mm, left=30mm, right=30mm]{geometry}

\usepackage{xcolor}
\theoremstyle{plain}
\newtheorem{thm}{Theorem}[section]

\newtheorem{lem}[thm]{Lemma}
\newtheorem{prop}[thm]{Proposition}

\newtheorem{ques}[thm]{Question}

\theoremstyle{definition}

\newtheorem{rem}[thm]{Remark}

\numberwithin{equation}{section}

\newcommand{\mb}{\mathbb}
\newcommand{\mc}{\mathcal}

\def \a{\alpha}   \def \d{\delta}
   \def \e{\epsilon} \def \ve{\varepsilon}
 \def \l{\lambda}

\allowdisplaybreaks

\renewcommand*{\backref}[1]{}
\renewcommand*{\backrefalt}[4]{\quad \tiny
  \ifcase #1 (\textbf{NOT CITED.})%
  \or    (Cited on page~#2.)%
  \else   (Cited on pages~#2.)%
  \fi}

\makeatletter

\def\MRbibitem{\@ifnextchar[\my@lbibitem\my@bibitem}

\def\mybiblabel#1#2{\@biblabel{{\hyperref{http://www.ams.org/mathscinet-getitem?mr=#1}{}{}{#2}}}}

\def\myhyperanchor#1{\Hy@raisedlink{\hyper@anchorstart{cite.#1}\hyper@anchorend}}

\def\my@lbibitem[#1]#2#3#4\par{%
  \item[\mybiblabel{#2}{#1}\myhyperanchor{#3}\hfill]#4%
  \@ifundefined{ifbackrefparscan}{}{\BR@backref{#3}}%
  \if@filesw{\let\protect\noexpand\immediate
    \write\@auxout{\string\bibcite{#3}{#1}}}\fi\ignorespaces%
}

\def\my@bibitem#1#2#3\par{%
  \refstepcounter\@listctr
  \item[\mybiblabel{#1}{\the\value\@listctr}\myhyperanchor{#2}\hfill]#3%
  \@ifundefined{ifbackrefparscan}{}{\BR@backref{#2}}%
  \if@filesw\immediate\write\@auxout
    {\string\bibcite{#2}{\the\value\@listctr}}\fi\ignorespaces%
}

\makeatother

\begin{document}

\author{Wen Huang} \address[Wen Huang] {Wu Wen-Tsun Key Laboratory of Mathematics, USTC, Chinese Academy of Sciences and Department of Mathematics, \\
 University of Science and Technology of China,\\
  Hefei, Anhui, China}
\email[W. Huang]{wenh@mail.ustc.edu.cn}

 \author{Zeng Lian} \address[Zeng Lian] {College of Mathematical Sciences\\ Sichuan University\\
    Chengdu, Sichuan, 610016, China} \email[Z. Lian]{lianzeng@scu.edu.cn, zenglian@gmail.com}

\author{Xiao Ma} \address[Xiao Ma] {  Wu Wen-Tsun Key Laboratory of Mathematics, USTC, Chinese Academy of Sciences and Department of Mathematics, \\
 University of Science and Technology of China,\\
  Hefei, Anhui, China}
\email[X. Ma]{xiaoma@ustc.edu.cn}

\author{Leiye Xu} \address[Leiye Xu] { Wu Wen-Tsun Key Laboratory of Mathematics, USTC, Chinese Academy of Sciences and Department of Mathematics, \\
 University of Science and Technology of China,\\
  Hefei, Anhui, China}
\email[L. Xu]{leoasa@mail.ustc.edu.cn}

\author{Yiwei Zhang} \address[Yiwei Zhang] { School of Mathematics and Statistics, Center for Mathematical Sciences, Hubei Key Laboratory of
Engineering Modeling and Scientific Computing,\\ Huazhong University of Sciences and Technology,\\
Wuhan 430074, China}
\email[Y. Zhang]{yiweizhang831129@gmail.com}

\title[]{Ergodic optimization theory for a class of typical maps}

\thanks{Huang is partially supported by NSF of China (11431012,11731003). Lian is partially supported by NSF of China (11725105,11671279). Xu is partially supported by NSF of China (11801538, 11871188). Zhang is partially supported by NSF of China (117010200,11871262).}

\begin{abstract}
In this article, we consider the weighted ergodic optimization problem of a class of dynamical systems $T:X\to X$ where $X$ is a compact metric space and $T$ is Lipschitz continuous. We show that once  $T:X\to X$ satisfies both the {\em Anosov shadowing property }({\bf ASP}) and the {\em Ma\~n\'e-Conze-Guivarc'h-Bousch property }({\bf MCGBP}),  the minimizing measures of generic H\"older observables are unique and supported on a periodic orbit.
Moreover, if $T:X\to X$ is a subsystem of  a dynamical system $f:M\to M$ (i.e. $X\subset M$ and $f|_X=T$) where $M$ is  a compact smooth manifold, the above conclusion holds for $C^1$ observables.

Note that a broad class of classical dynamical systems satisfies both ASP and MCGBP, which includes {\em Axiom A attractors, Anosov diffeomorphisms }and {\em uniformly expanding maps}. Therefore, the  open problem proposed by Yuan and Hunt in \cite{YH} for $C^1$-observables is  solved consequentially.
\end{abstract}

\maketitle

\parskip 0.4cm
\section{Introduction}\label{S:Introduction}

\noindent {\bf Context and motivation:}
Ergodic optimization theory mainly studies the problems relating to   minimizing (or maximum) orbits, minimizing (or maximum) invariant measures and minimizing   (or maximum) ergodic averages. \\
This theory has strong connection with other fields, such as Anbry-Mather theory \cite{Contreras_Mane,Mane} in Lagrangian Mechanics; ground state theory \cite{BLL} in thermodynamics formalism and multifractal analysis; and controlling chaos \cite{OGY90,SGOY93} in control theory. 

Let $(X,T)$ be a  dynamical system and $u$ be a real-valued function on $X$. For a given orbit of $T$, say $\mc O=\{T^ix\}_{i=0,1,\cdots}$, the time average of $u$ along $\mc O$ is usually defined by $\lim_{n\to\infty}\frac1n\sum_{i=0}^{n-1}f(T^ix)$ when it converges. An orbit is called $u$-minimizing if the time average of $u$ along the orbit is less than along any other orbits. \\
In a reasonably general case, $X$ is assumed compact and $T$, $u$ are assumed continuous. Thereafter, $\mathcal{M}(X,T)$, the set of all $T$-invariant Borel probability measures on $X$, becomes a non-empty convex and compact topological space with respect to weak$^*$ topology. It is well known that, by Birkhoff ergodic theorem, for a given $\mu\in \mc M(X,T)$, the time averages of $u$ are well defined for orbits initiated on $\mu$-a.e. $x$. Moreover, when $\mu$ is ergodic, time averages of $u$ become constant ($\mu$-a.e.) and equal to the space average $\int_Xud\mu$ which is also called the ergodic average of $u$ with respect to $\mu$. Denote by  $\mathcal{M}^{e}(X,T)\subset\mathcal{M}(X,T)$  the collection of ergodic measures, which is the set of  the extremal points of $\mathcal{M}(X,T)$.
By the compactness and convexity  of $\mc M(X,T)$, the minimizing (or maximizing) ergodic average of $u$ always exists and can be achieved by some ergodic measure which is called $u$-minimizing measure. 

To understand minimizing measures is the main task of the classical ergodic optimization theory. This kind of problem involves three main factors: complexity of the systems, regularity of the observable functions, and complexity of the minimizing measures. The well known Meta-Conjecture 
states that {\em when the systems is chaotic then the minimizing measures of generic typical observables have low complexity}.  As a special case of the Meta-Conjecture, a more precise conjecture, the Typical Periodic Optimization (TPO) Conjecture,  is proposed by Yuan and Hunt (\cite{YH}, 1999) first and later in a general form, which is focusing on a special class of chaotic systems including  uniformly hyperbolic systems and uniformly expanding maps, and is conjecturing  that  for a suitably continuous real-valued observable space $V$ (e.g. Lipschitz observable space),  $V_{per}$ contains an open and dense subset, where $V_{per}$ is the subspace of $V$ such that  for each $u\in V_{per}$ the set of $u$-minimizing measures contains at least a periodic measure. TPO Conjecture is one of the fundamental questions raised in the field of ergodic optimization theory, which has attracted sustained attention and yielded considerable results for last two decades. 

The first attempt towards the TPO Conjecture is due to Contreras, Lopes and Thieullen \cite{CLT}. They considered the case where $T$ is smooth, orientation-preserving, uniformly expanding map of the circle, and $V$ is the $\a$-H\"older functions space $C^{0,\a}$, and derived a weaker version result. \\
Afterwards, there is a series of works for the case that $(X,T)$ is subshifts of finite type: for example, Bousch \cite{Bousch_Walters} proved that the TPO Conjecture holds for Walter functions; Quas and Siefken \cite{QS} asserted the validity of the conjecture for the case where $T$ is a full shift and $V$ is the space of "super continuous" functions; Bochi and Zhang (\cite{BZ}) solved the case when $T$ is a one-side shift on two symbols and $V$ is a space of functions with strong modulus of regularity; Morris \cite{Morris_entropy} proved that for a generic real-valued H\"older continuous function on a subshift of finite type, the maximizing measure must have zero entropy.  \\
 For more comprehensive survey for the classical ergodic optimization theory, we refer the readers to Jenkinson \cite{Jenkinson_survey, Jenkinson_newsurvey}, to Bochi \cite{Bochi} and to Baraviera, Leplaideur, Lopes \cite{BLL} for a historical perspective of the development in this area.\\
  In the existing literature, the best result towards the TPO Conjecture is obtained by Conteras \cite{Contreras_EO}, which built on work of \cite{Bousch_Mane, Morris_entropy, QS, YH}. The main result obtained in \cite{Contreras_EO}  states that for a uniformly expanding map the minimizing measures of (topological) generic Lipschitz observations are uniquely supported on periodic orbits.

The purpose of this paper is  to investigate validity of the TPO conjecture for a considerably broad class of  typical dynamical systems such as Anosov diffeomorphisms. Precisely,  systems  considered in this paper are assumed to be Lipschitz and satisfy the so called {\bf Anosov Shadow property} (abbr. ASP) and {\bf Ma\~n\'e-Conze-Guivarc'h-Bousch property} (abbr. MCGBP),  definitions of which are given in Section \ref{S:SetResults}. In the existing literature, {\em Axiom A attractors, Anosov diffeomorphisms} and {\em expanding maps} all satisfy both  {\bf ASP} and {\bf MCGBP}.


\noindent {\bf Summary of the main result:}
In this paper, we extend the validity of the TPO Conjecture on two scopes: systems and observables. To avoid tediousness, we summarize only part of the main result into the following theorem, the precise statement of which will be given in Section \ref{S:SetResults}.

\noindent{\bf Theorem A:} {\em For an Axiom A or uniformly expanding system and a (topologically) generic observable function  from H\"older function space, Lipschitz function space or $C^{1,0}$ function space (if well defined) the minimizing measure is unique and is supported on  a periodic orbit.}

This gives a positive answer to the conjecture proposed by Yuan and Hunt on 1999 (Conjecture 1.1 in \cite{YH}) for H\"older, Lipschitz, and $C^1$ cases.

\noindent {\bf Remarks on the techniques of the proof:}
A major ingredient of the proof in this paper is based on  a closing lemma due to Bressaud and Quas \cite{BQ}, which allow us to identify the suitable periodic orbits with "good shape". By the "good shape", roughly speaking, we mean that the periodic orbit should satisfy three properties:1) period of the orbit should not be too large; 2) points of the orbit should be distributed evenly, in another word, distance of distinguished points from the orbit should not be too small; 3) orbit should be close enough to a given invariant compact set.  Based on these periodic orbit with "good shape", one can construct a sequence of observables whose minimizing measures are unique and periodic converges to a arbitrarily given observable.  

Comparing with the proofs of  existing results such as Contreras' proof in \cite{Contreras_EO}, our proof follows a more direct way. For example, since Contreras' proof firstly went to an intermedia result of Morris  \cite{Morris_entropy} which states that the minimizing measures of generic Lipschitz observables have zero entropy, the methodology largely depends on the uniformly expanding property of the system;  In contrast, our proof is devoted to searching the target periodic orbits directly, which automatically avoids the entropy argument of the minimizing measures, thus the methodology is applicable to a broader class of systems beyond expanding maps. 


\noindent{\bf Organization of the paper:} The structure of the paper is organized as follows. In Section \ref{S:SetResults}, we give the main setting and notions, and state the main result; In Section \ref{S:ProofMainResult}, we give the proof of the main result (Theorem \ref{T:MainResult}); In Section \ref{S:CsCase}, we consider the case of observable functions with high regularity, for which some partial results and remaining questions are presented; In Appendix \ref{Sec-manelemma}, we  briefly explain why Anosov diffeomorphisms being {\bf MCGBP} for the sake of completeness.


\section{settings and results}\label{S:SetResults}
In this paper, we will study the typical optimization problem in weighted ergodic optimization theory which has strong connection with zero temperature limit. The current section is devoted to formulating the setting and stating the result. 

Let $(X,d)$ be a compact
metric space and $T : X \rightarrow X$ be a continuous map. Denote by $\mathcal{M}(X,T)$ the set of all $T$-invariant Borel probability measures on $X$, which is a non-empty convex and compact topological space with respect to weak$^*$ topology. Denote by  $\mathcal{M}^{e}(X,T)\subset\mathcal{M}(X,T)$  the ergodic measures, which is the set of  the extremal points of $\mathcal{M}(X,T)$.\\
Let $u:X\to\mathbb{R}$ and $\psi:X\to\mathbb{R}^{+}$ be continuous functions.  The quantity $\beta(u;\psi, X, T)$ defined by
\begin{equation}\label{equ_ratiomimimalerg}
\beta(u;\psi,X,T):=\min_{\nu\in\mathcal{M}(X,T)}\frac{\int u d\nu}{\int \psi d\nu},
\end{equation}
is called the \emph{ratio minimum ergodic average}, and any $\nu\in\mathcal{M}(X,T)$ satisfying $$\frac{\int u d\nu}{\int \psi d\nu}=\beta(u;\psi,X,T)$$
is called a \emph{$(u,\psi)$-minimizing measure}. Denote that $$\mathcal{M}_{\min}(u;\psi,X,T):=\left\{\nu\in\mathcal{M}(X,T):\frac{\int ud\nu}{\int \psi d\nu}=\beta(u;\psi,X,T)\right\}.$$
By compactness of $\mathcal{M}(X,T)$, and the continuity of the operator $\frac{\int ud(\cdot)}{\int \psi d(\cdot)}$, it directly follows that $\mathcal{M}_{min}(u;\psi,X,T)\neq\emptyset$, which contains at least one ergodic $(u,\psi)$-minimizing measure by ergodic decomposition.

For each real number $\eta\geq0$, we call a sequence $\{x_{i}\}_{i=0}^{n-1}\subset X$  \emph{a periodic $\eta$-pseudo-orbit} of $(X,T)$, if each $x_{i+1}$ belongs to an $\eta$-neighbourhood of $T(x_{i})$, for all $i=0,\cdots,n-1\mod n$. With this convention, we say that $(X,T)$ satisfies {\bf Anosov Shadow property} (abbr. ASP) and {\bf Ma\~n\'e-Conze-Guivarc'h-Bousch property} (abbr. MCGBP), if
\begin{enumerate}
	\item[A.]{\bf (ASP)} There are positive constants $\lambda,\delta,C,L$ such that
	\begin{enumerate}
		\item[(1).] For $n\in\mathbb{N}$ and $x,y\in X$ with $d(T^{i}x,T^{i}y)\le\delta$ for all $0\le i\le n$, one has for all $0\le k\le n$,
		$d(T^{k}x,T^{k}y)\le Ce^{-\lambda\min(k,n-k)}(d(x,y)+d(T^{n}x,T^{n}y)).$
		\item[(2).] For any $0\le\eta\le \delta$, $n\ge 1$ and a periodic $\eta$-pseudo-orbit $\{x_i\}_{i=0}^{n-1}$, there is a periodic orbit $\{T^{i}x\}_{i=0}^{m-1}$ with period $m$ such that $m|n$ and $d(x_i,T^{i}x)\le L\eta,~~\forall 0\leq i\leq n-1$.
	\end{enumerate}
	\item[B.]{\bf (MCGBP)} For any $0<\alpha\le 1$, there exists positive integer $K=K(\alpha)$ such that for all $u\in \mathcal{C}^{0,\alpha}(X)$, there is $v\in \mathcal{C}^{0,\alpha}(X)$ such that $$\bar u:=u_K-v\circ T^K+v-\beta(u;{X},T)\ge 0$$ where $u_K=\frac{1}{K}\sum_{i=0}^{K-1}u\circ T^i$ and $\beta(u;{X},T)=\min_{\nu\in\mathcal{M}(X,T)}\int u d\nu$.
\end{enumerate}
Here, $\a\in (0,1]$ and $\mathcal{C}^{0,\alpha}(X)$ is the space of $\alpha$-H\"{o}lder continuous real-valued function on $X$ endowed with the $\alpha$-H\"{o}lder norm $\|u\|_{\alpha}:=\|u\|_{0}+[u]_{\alpha}$, where $\|u\|_{0}:=\sup_{x\in X}|u(x)|$ is the super norm, and $[u]_{\alpha}:=\sup_{x\neq y}\frac{|u(x)-u(y)|}{d^{\alpha}(x,y)}$. Also note that when $\a=1$, $C^{0,1}(X)$ becomes the collection of all real valued Lipschitz continuous functions, and  $[u]_{1}$ becomes the minimum Lipschitz constant of $u$.

For the sake of completeness, we give a brief proof of Anosov diffeomorphisms satisfying {\bf MCGBP} in the Appendix Section \ref{Sec-manelemma}, while {\bf ASP} is a standard property of Anosov diffeomorphisms thus the proof of which is not repeated in this paper.

In summary, let $\mathfrak{C}$ be the set of triple $(X,T,\psi)$ satisfying the following properties:
\begin{itemize}
\item[H1)] $(X,d)$ is a compact metric space and $T:X\rightarrow X$ is Lipschitz continuous;
\item[H2)] $(X,T)$ satisfying  {\bf ASP} and {\bf MCGBP};
\item[H3)] $\psi:X\to \mb R^+$ are continuous.
\end{itemize}
The main results obtained in this paper is summarized in the following:
\begin{thm}\label{T:MainResult}
Suppose $(X,T,\psi)\in\mathfrak C$, then the following hold:
\begin{itemize}
\item[I)] For $\a\in (0,1]$, if $\psi\in \mc C^{0,\a}(X)$, then there exists an open and dense set $\mathfrak P\subset \mc C^{0,\a}(X)$ such that for any $u\in \mathfrak P$, $(u,\psi)$-minimizing measure is uniquely supported on a periodic orbit of $T$.
\item[II)] If $(X,T)$ is a sub-system of a dynamical system $(M,f)$ (i.e. $X\subset M$ and $T=f|_X$) and $\psi\in \mc C^{0,1}(M)$, where $M$ is a compact $\mc C^\infty$ Riemannian manifold , then there exists an open and dense set $\mathfrak P\subset \mc C^{1,0}(M)$ such that for any $u\in \mathfrak P$, the $(u|_X,\psi|_X)$-minimizing measure of $(X,T)$ is uniquely supported on a periodic orbit of $T$, where $\mc C^{1,0}(M)$ is the Banach space of continuous differentiable functions on $M$ endowed with the standard $C^1$-norm.
\end{itemize}
\end{thm}
\begin{rem}\label{R:NoDPAssump}
It is worth to point out that Theorem \ref{T:MainResult} only requires $(X,T)$ satisfy {\bf ASP} and {\bf MCGBP}, which means that, in particular, neither topological transitivity for Anosov diffeomorphisms (although it is conjectured that Anosov diffeomorphisms are always topological transitive) nor non-wandering property for Axiom A attractors are needed. { If in addition $X\supset \text{supp}(\mu)$ for all $\mu\in \mathcal{M}(M,f)$ in Theorem \ref{T:MainResult} (II), then for any $u\in \mathfrak P$, the $(u,\psi)$-minimizing measure of $(M,f)$ is also uniquely supported on a periodic orbit of $f$.}
\end{rem}


On the other hand, the reason of adding the nonconstant weight $\psi$ mainly lies in the studies on the zero temperature limit (or ground state) of the $(u,\psi)$-weighted equilibrium state for thermodynamics formalism, i.e., the measure $\mu_{u,\psi}\in\mathcal{M}(X,T)$, which satisfies
\begin{equation}\label{equ:weigthed equilibriumstate}
 \mu_{u,\psi}:=\arg\max\left\{h_{\nu}(T)+\frac{\int ud\nu}{\int\psi d\nu}:~\forall\nu\in\mathcal{M}(X,T)\right\},
\end{equation}
where $h_{\nu}$ is the Kolmogorov-Sinai entropy of $\nu$. Such weighted equilibrium state arises naturally in the studies of non-conformal multifractal analysis (e.g. high dimensional Lyapunov spectrum) for asymptotically (sub)additive potentials, see works \cite{BF,BCW,FH}. In fact, when the ground state exists, (i.e., the limit $\lim_{t\to+\infty}\mu_{tu,\psi}$ exists), then the limit formulates a special candidate of $(-u,\psi)$-minimizing measure.

\section{Proof of Theorem \ref{T:MainResult}}\label{S:ProofMainResult}
Before starting the proof, we introduce some notions at first for the sake of convenience.  For  $(X,T,\psi)\in\mc C$ and a continuous function $u:X\to\mathbb{R}$, define
\begin{equation}\label{E:Z_u,psi}
Z_{u,\psi}:=\overline{\cup_{\mu\in\mathcal{M}_{min}(u;\psi,{X},T)}supp(\mu)}.
\end{equation}
For $\a\in(0,1]$, a non-empty subset $Z$ of $X$ and a periodic orbit $\mc O$ of $(X,T)$, define \emph{the $\a$-deviation of $\mc O$ with respect to $Z$} by
$$d_{\a,Z}(\mc O)=\sum_{x\in \mc O}d^\a(x,Z),$$
where we recall that $d$ is the metric on $X$. Let $\l,\d,C,L$ be the constants as in {\bf ASP} and fix these notations. Define $D:{X}\times{X}\to[0,+\infty)$ by
$$D(x,y)=\left\{\begin{array}{ll}
\delta, & \text{ if } d(x,y)\ge \delta,\\
d(x,y), & \text{ if } d(x,y)< \delta.
\end{array}
\right.$$
By a periodic orbit $\mathcal{O}$ of $(X,T)$, the {\it gap} of $\mathcal{O}$ is defined by
\begin{equation}\label{E:DefDO}
D(\mathcal{O})=\left\{\begin{array}{ll}
\delta, & \text{ if } \sharp\mathcal{O}=1,\\
\min_{x,y\in\mathcal{O}, x\not=y}D(x,y), & \text{ if } \sharp\mathcal{O}>1.
\end{array}
\right.
\end{equation}

We will prove Part I) and  Part II) of Theorem \ref{T:MainResult} separately.

\subsection{Proof of Part I) of Theorem \ref{T:MainResult}}\label{S:Part I)}

The proof of Part I) of Theorem \ref{T:MainResult} mainly contains two steps:

\begin{itemize}
\item[Step 1.] We show how to construct periodic orbit $\mc O$ of $(X,T)$ to make the ratio $\frac{D^\a(\mc O)}{d_{\a,Z}(\mc O)}$ as large as needed. Such a periodic orbit will be a candidate to support the minimizing measures of observables nearby $u$.
\item[Step 2.] We show that for any given $u\in \mc C^{0,\a}(X)$ there exists an open set $\mc U$ of observables being arbitrarily close to $u$, whose minimizing measures are all supported on a single periodic orbit $\mc O$.
\end{itemize}
The main aim of Step 1 can be summarized into the following Proposition:
\begin{prop}\label{P:GoodPer}
Let $(X,T)$ satisfy H1) and {\bf ASP}, $u : X\rightarrow \mathbb{R}$ and $\psi: X \rightarrow \mathbb{R}^+$ are continuous.
Then for any $\a\in(0,1]$,  a  given $\tilde L>0$ and $T$-forward-invariant non-empty subset $Z\subset X$ (i.e. $T(Z)\subset Z$), there exists an periodic orbit $\mc O$ of $(X,T)$ such that
\begin{align}\label{E:GapCond0}
	\frac{D^\alpha(\mathcal{O})}{d_{\alpha,Z}(\mathcal{O})}> \hat L.
\end{align}
\end{prop}
Similarly, the main aim of Step 2 can be summarized into the following Proposition:
\begin{prop}\label{P:OpenDencePerMin}
Let $(X,T,\psi)\in\mc C$ (satisfying H1), H2) and H3)) and $\psi, u \in \mathcal{C}^{0,\alpha}({X})$ for some $\alpha\in (0, 1]$. Then for  any $\varepsilon>0$, there exist $\hat L,\hat \d>0$ which depend on $\ve,\a, u, \psi$ and system constants only such that the following holds:  If there is a  periodic orbit  $\mathcal{O}$ of $(X,T)$ satisfying
\begin{align}\label{E:GapCond}
	\frac{D^\alpha(\mathcal{O})}{d_{\alpha,Z_{u,\psi}}(\mathcal{O})}> \hat L,
	\end{align}
 then the observable function $u_{\e,h}:=u+\varepsilon d^\alpha(\cdot,\mathcal{O}) +h$ has a unique minimizing measure $$\mu_{\mathcal{O}}:=\frac{1}{\sharp\mathcal{O}}\sum_{x\in\mathcal{O}}\delta_{x}$$
 whenever $h\in \mathcal{C}^{0,\alpha}({X})$ with $\|h\|_\alpha<10\e \text{ and }
\|h\|_0<\frac{D^\a(\mc O)}{\sharp \mc O}\cdot \hat \d.$
\end{prop}

It is clear that the collection of $u_{\e,h}$ in Proposition \ref{P:OpenDencePerMin} forms an non-empty open subset of $\mc C^{0,\a}(X)$, which is about $\varepsilon$-apart from $u$. Since $\ve$ can be taken arbitrarily small and the existence of periodic orbit satisfying (\ref{E:GapCond}) are guaranteed by Proposition \ref{P:GoodPer}, Part I) of Theorem \ref{T:MainResult} follows. Thus it remains to prove Proposition \ref{P:GoodPer} and \ref{P:OpenDencePerMin}.


\subsubsection{Proof of Proposition \ref{P:GoodPer}}\label{S:ProofPropGoodPer}

At first, we introduce a lemma giving a quantified estimate of the denseness of periodic orbits, which can be viewed as a version of Quas and Bressaud's periodic approximation lemma.

\begin{lem}\label{lemma-3}
Let $(X,T)$ be a dynamical system satisfying  H1) and {\bf ASP}, $Z$ be a nonempty $T$-forward-invariant subset of $X$. Then for all $\a\in(0,1]$ and  $k>0$, one has
$$\lim_{n\to\infty}n^k\min_{\mathcal{O}\in \mathcal{O}^n}d_{\alpha,Z}(\mathcal{O})=0,$$
where $\mathcal{O}^n$ denote the collection of all periodic orbits of $(X,T)$ with period not larger than  $n$
\end{lem}

\begin{proof}
We follow the arguments in \cite{BQ}. Before going to the proof, we need to state two technical results first. The following lemma is Lemma 5 of \cite{BQ}.

Let $\Sigma_n=\{0,1,2,\cdot,n-1\}^{\mathbb{N}}$ and $\sigma$ is a shift on $\Sigma_n$. Assume $F$ is a subset of $\bigcup_{i\ge 1}\{0,1,2,\cdot,n-1\}^{i}$, then the subshift with forbidden $F$ is noted by $(Y_F,\sigma)$ where
$$Y_F=\{x\in \{0,1,2,\cdot,n-1\}^{\mathbb{N}},w\text{ does not appear in } x\text{ for all }w\in F\}.$$
	\begin{lem}[\cite{BQ}]\label{lemma-1} Suppose that $(Y,\sigma)$ is a shift of  finite type $($with forbidden words of length 2$)$ with $M$ symbols and entropy $h$. Then $(Y,\sigma)$ contains a periodic point of period at most $1+Me^{(1-h)}$.
\end{lem}

\begin{lem}\label{lemma-2}
Let $(X,T)$ be a dynamical system satisfying H1) and {\bf ASP}, and $Z$ be a nonempty subset of $X$.  Then for any $0<\alpha\le 1$, $0<\eta\le \delta$, $n\ge 0$, and periodic $\eta$-pseudo-orbit $\widetilde{\mathcal{O}}$ of $(X,T)$ with period $n$, there exists a periodic orbit $\mathcal{O}$ of $(X,T)$ with period $m$ such that $m|n$ and
	$$d_{\alpha,Z}(\mathcal{O})\le d_{\alpha,Z}(\widetilde{\mathcal{O}_n})+n(L\eta) ^\alpha,$$
	where $\delta,L$ are the constants as in {\bf ASP}.
\end{lem}

\begin{proof}For $n\ge 1$, assume $\widetilde{\mathcal{O}}=\{x_i\}_{i=0}^{n-1}$ is a periodic $\eta$-pseudo-orbit with period $n$. By {\bf ASP}, there is a periodic orbit $\mathcal{O}=\{x,Tx,\cdots,T^{m-1}x\}$ such that $m|n$ and $d(x_i,T^ix)\le L\eta$ for $0\le i\le n-1$. Therefore, for $0<\alpha\le1$, one has
	$$d_{\alpha,Z}(\mathcal{O})=\sum_{i=0}^{m-1}d^\alpha(T^ix,Z)\le \sum_{i=0}^{n-1}(d(x_i,Z)+L\eta)^\alpha\le d_{\alpha,Z}(\widetilde{\mathcal{O}})+n(L\eta)^\alpha.$$
	This ends the proof.
	\end{proof}

Now, we are ready to prove Lemma \ref{lemma-3}.

Fix $\alpha, k, Z$ as in the lemma and  $\lambda,\delta,C,L$ are the constants as in {\bf ASP}. Let  $\mathcal{P}=\{P_1,P_2,\cdots,P_m\}$ be a finite partition of $X$ with diameter smaller than $\delta$. For $x\in X$, $\widehat{ x}\in\{1,2,3,\cdots,m\}^{\mathbb{N}}$ is defined by
	$$\widehat{x}(n)=j\text{ whenever } T^nx\in P_j\text{ and } n\in\mathbb{N}.$$
	Denote $\widehat{Z}=\{\widehat{x}: x\in Z\}$ and $W_n$ is the collection of
	length $n$ strings that appear in $\widehat{Z}$. Then  $K_n :=e^{-nh}\sharp W_n$ grows at a subexponential rate, i.e.
	$$\lim_{n\to\infty}\frac{\log K_n}n= 0,$$
	where $h=h_{top}(\overline{\widehat{Z}},\sigma)$.  Denote
	$$Y_n=\{y_0y_1y_2\cdots\in W_n^\mathbb{N}:y_i\in W_{n} \text{ and } y_iy_{i+1}\in W_{2n}\text{ for all }i\in\mathbb{N}\}.$$
	Let $(Y_n,\sigma_n)$ be the 1-step shift of finite type $W_n$. Then $(\overline{\widehat{Z}},\sigma^n)$can be considered as  a subsystem of $(Y_n,\sigma_n)$. Hence $$h_{top}(Y_n,\sigma_n)\ge h_{top}(\overline{\widehat{Z}},\sigma^n)=nh.$$ Thus from Lemma \ref{lemma-1}, the shortest periodic orbit in $Y_n$ is at most $1+e^{1-nh}\sharp W_n=1+eK_n$.  Denote one of the shortest periodic orbit in $Y_n$ by $z_1z_2\cdots z_{p_n}z_1z_2\cdots$ for some $p_n\le 1+eK_n$ and $z_i\in W_n, i=1,2,\cdots,p_n$.
	
	Now we construct a periodic pseudo-orbit in $Z$. For $i=1,2,\cdots p_n$, there is $x_i\in Z$ such that the leading length $2n$ string of $\widehat{x_i}$ is $z_iz_{i+1}$ (Note $z_{{p_n}+1}=z_1$). Hence, $\widehat{T^nx_i}$ and $\widehat{x_{i+1}}$ have the same leading length $n$ string which implies $d(T^{n+j}x_i,T^{j}x_{i+1})<\delta$ $($Note $x_{p_n+1}=x_1$$)$ for $j=0,1,2,\cdots,n-1$.  By {\bf ASP},
	\begin{align*}
	d(T^{n+[\frac{n}{2}]}x_i,T^{[\frac{n}{2}]}x_{i+1})&<Ce^{-\lambda\min([\frac{n}{2}],n-1-[\frac{n}{2}])}\cdot(d(T^{n}x_i,x_{i+1})+d(T^{2n-1}x_i,T^{n-1}x_{i+1}))\\
	&\le 2\delta Ce^{-\lambda([\frac{n}{2}]-1)}.
	\end{align*}
	Therefore, we select the periodic $2\delta Ce^{-\lambda([\frac{n}{2}]-1)}$-pseudo-orbit $\widetilde{\mathcal{O}_n}$ in $Z$ with periodic $n{p_n}$ by
	$$\{T^{[\frac{n}{2}]}x_1,T^{[\frac{n}{2}]+1}x_1,\cdots,T^{n+[\frac{n}{2}]-1}x_1,T^{[\frac{n}{2}]}x_2,\cdots,T^{n+[\frac{n}{2}]-1}x_2,T^{[\frac{n}{2}]}x_3,\cdots,T^{n+[\frac{n}{2}]-1}x_{p_n}\}.$$

	By {\bf ASP}, while $n$ is sufficiently large, we have a periodic orbit $\mathcal{O}_n$ with period $m_n$, such that $m_n|n{p_n}$
	$$d_{\alpha,Z}(\mathcal{O}_n)\le d_{\alpha,Z}(\widetilde{\mathcal{O}_n})+n{p_n}(2\delta CL)^\alpha e^{-\lambda\alpha([\frac{n}{2}]-1)}=n{p_n}(2\delta C L)^\alpha e^{-\lambda\alpha([\frac{n}{2}]-1)},$$
where we used lemma \ref{lemma-2}.	Since $p_n\le 1+eK_{n}$ and $K_n$ grows at a subexponential rate, we obtain
\begin{equation*}
\begin{aligned}
\limsup_{n\to\infty}n^k\min_{\mathcal{O}\in \mathcal{O}^n}d_{\alpha,Z}(\mathcal{O})
&\le \limsup_{n\to\infty} (\max\{ip_i:{1\le i\le n+1}\})^k\cdot n{p_n} \delta ^\alpha e^{-\lambda \alpha([\frac{n}{2}]-1)}=0.
\end{aligned}
\end{equation*}
This ends the proof.
	\end{proof}

Let $(X,T)$ satisfy H1) and ASP, $u : X\rightarrow \mathbb{R}$ and $\psi: X \rightarrow \mathbb{R}^+$ are continuous. Given
$\a\in(0,1]$,   $\tilde L>0$ and  a $T$-forward-invariant non-empty subset $Z\subset X$. Now, we are ready to construct the required periodic orbit in Proposition \ref{P:GoodPer} satisfying (\ref{E:GapCond0}). Before the rigorous proof, we firstly introduce the idea of the construction in a vague way: One can start with a periodic orbit $\mc O_0$ with long enough period  $n$ and a good approximation to $Z$ (say $d_{\a,Z}(\mc O_0)<n^{-k} $ for some large $k$); Once the gap of $\mc O_0$, $D(\mc O_0)$, is too small to meet the requirement, $\mc O_0$ can be decomposed into two pseudo periodic orbits, one of which has at most half of the original period $n$; Such pseudo orbits will provide a nearby periodic orbit with same period by {\bf ASP}, say $\mc O_1$;  One can show that the ratio $\frac{d_{\a,Z}(\mc O_1)}{ d_{\a,Z}(\mc O_0)}$ is bounded by a constant depending on system constants and $\tilde L$ only rather than dependending on $d_{\a,Z}(\mc O_0)$; Note that the operation of decomposing periodic orbits into periodic orbits with period halved can be done at most $\log_2 n$ times; Therefore, by adjusting the largeness of $n,k$, the above process will end at either a periodic orbit meet the requirement of Proposition \ref{P:GoodPer} or a fixed point, both of which will clearly accomplish the proof.

Let $C,L,\l,\d$ be as in {\bf ASP}. Take $k\in \mb N$ large enough, on which the condition will be proposed later. By Lemma \ref{lemma-3}, there exists a periodic orbit $\mc O_0$ of $(X,T)$ with period $n$ large enough such that
\begin{equation}\label{E:c-0}
d_{\a,Z}(\mc O_0)<\tilde L_0n^{-k}\ll \d,
\end{equation}
where $\tilde L_0=1$. If $D^\a(\mc O_0)>\tilde Ld_{\a,Z}(\mc O_0)$, the proof is done. Otherwise, one has that
\begin{equation}\label{E:c-1}
D^\a(\mc O_0)\le \tilde Ld_{\a,Z}(\mc O_0)<\tilde L \tilde L_0 n^{-k},
\end{equation}
which is required to be smaller than $\d^\a$ by choosing $n,k$ large enough. Therefore, there are $y\in \mc O_0$ and $1\le n_1\le n-1$ such that
$$d(y,T^{n_1}y)< (\tilde L \tilde L_0)^{\frac1\a}n^{-\frac k\a}< \d.$$
We split the periodic orbit $\mc O_0$ into two pieces of orbit by
\begin{align*}
&\mc Q^0_0=\{y,Ty, \cdots, T^{n_1-1}y\};\\
&\mc Q^1_0=\{T^{n_1}y,T^{n_1}y, \cdots, T^{n-1}y\}.
\end{align*}
Note that each of the above segment of orbit induces a $\d$-pseudo periodic orbit, and moreover, period of one such $\d$-pseudo periodic orbit does not exceed $\frac n2 $. Without losing any generality, we assume that $n_1\le \frac n2$.

By {\bf ASP}, there exists a periodic orbit
$$\mc O_1=\{z_1,Tz_1,\cdots, T^{m_1-1}z_1\}$$
such that $T^{m_1}z=z$, $m_1|n_1$ and $d(T^iy, T^iz_1)\le L(\tilde L \tilde L_0)^{\frac1\a}n^{-\frac k\a}$ for all $0\le i\le n_1-1$. Therefore, by {\bf ASP} again,  for all $0\le i\le n_1-1$,
$$d(T^iy, T^iz_1)\le Ce^{-\l\min\{i,n_1-i\}}2L(\tilde L \tilde L_0)^{\frac1\a}n^{-\frac k\a},$$
which $L(\tilde L \tilde L_0)^{\frac1\a}n^{-\frac k\a}$ is required to be smaller than $\d$, that is, $L^{\a}(\tilde L \tilde L_0)n^{-k}<\d^{\a}$ by choosing $n,k$ large enough.

Hence,
\begin{align}\begin{split}\label{E:c-2}
d_{\a,Z}(\mc O_1)\le&d_{\a,Z}(\mc Q^0_0)+\sum_{i=0}^{n_1-1}\left(Ce^{-\l\min\{i,n_1-i\}}2L\right)^\a(\tilde L \tilde L_0)n^{-k}\\
\le&d_{\a,Z}(\mc O_0)+\frac{2(2CL)^\a}{1-e^{-\l\a}}(\tilde L \tilde L_0) n^{-k}\\
<&\tilde L_1 n^{-k},
\end{split}
\end{align}
where $\tilde L_1=(1+\frac{2(2CL)^\a}{1-e^{-\l\a}}\tilde L)\tilde L_0= 1+\frac{2(2CL)^\a}{1-e^{-\l\a}}\tilde L$.

If $D^\a(\mc O_1)> \tilde L d_{\a,Z}(\mc O_1)$, the proof is done. Otherwise, one repeats the above operation to get another periodic orbit $\mc O_2$ with period $\le \frac{n}{4}$. Note that, in this case, in order to make the above process repeatable one only need $$\tilde L\tilde L_1n^{-k}<\d^\a \text{ and }L^{\a}\tilde L\tilde L_1n^{-k}<\d^\a,$$ which is doable by choosing $n,k$ large enough. Suppose the above operation can be executed $m$ times resulting at a periodic orbit $\mc O_m$. Then, by applying the same argument inductively, one has that
$$d_{\a,Z}(\mc O_m)< \tilde L_m n^{-k},$$
where $\tilde L_m=(1+\frac{2(2CL)^\a}{1-e^{-\l\a}}\tilde L)\tilde L_{m-1}=\cdots={\tilde L_1}^m$.
Since every operation will (at least) halve the period of the resulting periodic orbit, such process has to end before $\left(\left[\frac{\log n}{\log 2}\right]+1\right)$-th operation.  In order to make each operation doable, one only need $n,k$ satisfying the following condition
$$\tilde L\tilde L_1^{\frac{\log n}{\log 2}+1}n^{-k}=\tilde L \tilde L_1 n^{-k+\frac{\log \tilde L_1}{ \log 2}}<\d^\a \text{ and }L^{\a}\tilde L\tilde L_1^{\frac{\log n}{\log 2}+1}n^{-k}=L^{\a}\tilde L \tilde L_1 n^{-k+\frac{\log \tilde L_1}{ \log 2}}<\d^\a.$$
Note that, after the last operation being executed, there are two possible cases: the resulting periodic orbit either meet the requirement of Proposition \ref{P:GoodPer} or is a fixed point of $T$. In the second case,  by the definition of $D(\mc O)$ (\ref{E:DefDO}), requirement of  Proposition \ref{P:GoodPer} is also met once one additionally choose $n,k$ to satisfy that ${\tilde L_1}^{\frac{\log n}{\log 2}+1}n^{-k}<\delta^{\a}$, that is, ${\tilde L_1} n^{-k+\frac{\log \tilde L_1}{ \log 2}+1}<\delta^{\a}$.

The proof of Proposition \ref{P:GoodPer} is completed.


\subsubsection{ Proof of Proposition \ref{P:OpenDencePerMin}}\label{S:ProofProp2}


Before going to the proof of Proposition \ref{P:OpenDencePerMin}, we need to introduce a technical lemma and some notions  that play important roles in later proof.
\begin{lem}\label{revael} 
Let $(X,T)$ be a dynamical system satisfying H1) and  {\bf MCGBP}.  Then for all $0<\alpha\le 1$, strictly positive $\psi\in \mathcal{C}^{0,\alpha}({X})$ and $u\in \mathcal{C}^{0,\alpha}({X})$, there is $v\in \mathcal{C}^{0,\alpha}({X})$ such that
	\begin{itemize}
		\item[(1)] $u_K-v\circ T^K+v-\beta(u;\psi,{X},T)\psi_K\ge 0;$
		\item[(2)] $Z_{u,\psi}\subset\{x\in {X}:(u_K-v\circ T^K+v-\beta(u;\psi,{X},T)\psi_K)(x)=0\},$
	\end{itemize}
where $K=K(\alpha)$ is the natural number as in {\bf MCGBP} and $Z_{u,\psi}$ is given by (\ref{E:Z_u,psi}).
	\end{lem}
	\begin{proof}(1). By {\bf MCGBP}, we only need to show that $\beta(u-\beta(u;\psi,{X},T)\psi;X,T)=0$, that is,
	$$\min_{\mu\in\mathcal{M}({X},T)} \int u-\beta(u;\psi,{X},T)\psi d\mu=0.$$
It is immediately  from the fact
	$$\min_{\mu\in\mathcal{M}({X},T)}\frac{\int ud\mu}{\int \psi d\mu}=\beta(u;\psi,{X},T),$$
	where we used the assumption $\psi$ is strictly positive.
	
	(2).  By a probability measure $\mu\in\mathcal{M}_{min}(u;\psi,{X},T)$, we have
	$$\int u_K-v\circ T^K+v-\beta(u;\psi,{X},T)\psi_Kd\mu=\int u-\beta(u;\psi,{X},T)\psi d\mu=0.$$
	Combining (1) and the fact $u_K-v\circ T^K+v-\beta(u;\psi,{X},T)\psi_K$ is continuous, we have
	$$supp(\mu)\subset \{x\in{X}:(u_K-v\circ T^K+v-\beta(u;\psi,{X},T)\psi_K)(x)=0\}.$$
	Therefore, by the continuity of $u_K-v\circ T^K+v-\beta(u;\psi,{X},T)\psi_K$, one has that  $$Z_{u,\psi}=\overline{\cup_{\mu\in\mathcal{M}_{min}(u;\psi,{X},T)}supp(\mu)}\subset\{x\in{X}:(u_K-v\circ T^K+v-\beta(u;\psi,{X},T)\psi_K)(x)=0\}.$$ This ends the proof.
	\end{proof}

\begin{rem}\label{rem-reveal}
For convenience, in the following text, if we need to use lemma \ref{revael}, we use $\bar u$ to represent
	$u_K-v\circ T^K+v-\beta(u;\psi,{X},T)\psi_K$ for short. Then, $\bar u\ge 0$ and $Z_{u,\psi}\subset\{x\in{X}:\bar u(x)=0\}.$
	\end{rem}

Fix $\varepsilon, \alpha,\psi,u$ as in Proposition \ref{P:OpenDencePerMin},  $K,\bar u$ as in remark \eqref{rem-reveal}, and $C$, $\delta$ as in {\bf ASP}. By remark \ref{rem-reveal}, one has that
	$$\bar u\ge0\text{ and }Z_{u,\psi}\subset\{x\in{X}:\bar u(x)=0\}.$$
	In stead of investigating the minimizing measure of $\bar u+\varepsilon d^\alpha(\cdot,\mathcal{O})+h$, we consider a modified observable $G:=\bar u+\varepsilon d^\alpha(\cdot,\mathcal{O})+h-a_\mathcal{O}\psi_K$ which will provide more conveniences, where
	$$a_\mathcal{O}:=\frac{\sum_{y\in\mathcal{O}}\left(\bar u(y)+\varepsilon d^\alpha(y,\mathcal{O})+h(y)\right)}{\sum_{y\in\mathcal{O}}\psi_K(y)}.$$
	Clearly $\int Gd\mu_\mathcal{O}=0$.
	 Note that, by the definition of $u_K\left(:=\frac1K\sum_{i=0}^{K-1}u\circ T^i\right)$, for all $\mu\in\mathcal{M}({X},T)$, one has that
	\begin{align*}\frac{\int u+\varepsilon d^\alpha(\cdot,\mathcal{O}) +hd\mu}{\int \psi d\mu}&=\frac{\int u_K+\varepsilon d^\alpha(\cdot,\mathcal{O}) +hd\mu}{\int \psi_K d\mu}\\
	&=\frac{\int \bar u+\varepsilon d^\alpha(\cdot,\mathcal{O}) +hd\mu}{\int \psi_K d\mu}+\beta(u;\psi,{X},T)\\
	&=\frac{\int Gd\mu}{\int \psi d\mu}+a_\mathcal{O}+\beta(u;\psi,{X},T),
	\end{align*}
	where we recall that $\beta(u;\psi,{X},T)$ is the minimum ergodic average given by (\ref{equ_ratiomimimalerg}).

Then, in order to show that $\mu_{\mathcal{O}}\in\mathcal{M}_{min}(u+\varepsilon d^\alpha(\cdot,\mathcal{O})+h;\psi,{X},T)$,
it is enough to show that $\mu_{\mathcal{O}}\in\mathcal{M}_{min}(G;\psi,{X},T)$. Since $\psi$ is strictly positive and $\int Gd\mu_\mathcal{O}=0$, it is enough to show that
\begin{equation}\label{E:PositivityIntG}
\int Gd\mu>0\text{ for all }\mu\in\mathcal{M}^e({X},T)\setminus\{ \mu_{\mc O}\}.
\end{equation}

Denote that
$$L_{\mc O}=\frac{D^\alpha(\mathcal{O})}{d_{\alpha,Z_{u,\psi}}(\mathcal{O})}.$$
	
Thus one has an equivalent statement of Proposition \ref{P:OpenDencePerMin} under the same setting as the following:	
\begin{lem}\label{L:PositivityIntG}
There exist $\hat L,\hat \d>0$ which depend on $\ve,\a, u, \psi$ and system constants only such that if $L_{\mc O}>\hat L$, $\|h\|_\a\le 10\e$ and $\|h\|_0<\frac{D^\a(\mc O)}{\sharp \mc O}\cdot \hat \d$, then (\ref{E:PositivityIntG}) holds.
\end{lem}
\begin{proof}
Put $Area_1:=\left\{x\in{X}:d(x,\mathcal{O})\le\left(\frac{|a_\mathcal{O}|\|\psi\|_0+\|h\|_0}{\varepsilon}\right)^{\frac{1}{\alpha}}\right\}.$
Note that
	\begin{align*}
    \begin{split}
	|a_\mathcal{O}|
	&=\left|\frac{\sum_{y\in\mathcal{O}}\left(\bar u(y)+h(y)\right)}{\sum_{y\in\mathcal{O}}\psi_K(y)}\right|\\
		&\le\frac{\sum_{y\in\mathcal{O}}\left(\|\bar u\|_\alpha d^\alpha(y,Z_{u,\psi})+\|h\|_0\right)}{\sum_{y\in\mathcal{O}}\psi_{min}}\\
	&= \frac{\|\bar u\|_\alpha d_{\alpha,Z_{u,\psi}}(\mathcal{O})}{\sharp\mathcal{O}\psi_{min}}+\frac{\|h\|_0}{\psi_{min}}.
		\end{split}\end{align*}
Hence
\begin{align*}
\frac{\left(\frac{|a_\mathcal{O}|\|\psi\|_0+\|h\|_0}{\varepsilon}\right)^{\frac{1}{\alpha}} }{\frac{D(\mc O)}{2}}
&\le 2\left( \frac{\|\bar u\|_\alpha}{\epsilon \sharp\mathcal{O}\psi_{min}}\frac{1}{L_{\mc O}}+\frac{\|h\|_0}{\epsilon\psi_{min} D^\alpha(\mc O)}\right)^{\frac{1}{\alpha}}
\end{align*}
Particularly, when $L_{\mc O}>\frac{2(2Lip_T)^\alpha \|\bar u\|_\alpha}{\epsilon \sharp\mathcal{O}\psi_{min}}$ and $\|h\|_0<\frac{\epsilon\psi_{min} }{2(2Lip_T)^\alpha}D^\alpha(\mc O)$, one has
\begin{align*}
\left(\frac{|a_\mathcal{O}|\|\psi\|_0+\|h\|_0}{\varepsilon}\right)^{\frac{1}{\alpha}}<\frac{D(\mathcal{O})}{2Lip_T},
\end{align*}
where 
\begin{align}\label{lip-definition}
Lip_T=\left\{\begin{array}{ll}
1, & \text{ if } \sharp T(X)=1,\\
\max\left\{1,\sup_{x\neq y}\frac{d(Tx,Ty)}{d(x,y)}\right\}, & \text{ if } \sharp T(X)>1.
\end{array}
\right.
\end{align}
Thus in the choice of $\hat L$ and $\hat \d$, we will require  $\hat L\ge \frac{2(2Lip_T)^\alpha \|\bar u\|_\alpha}{\epsilon \psi_{min}} $ and  $\hat \d\le \frac{\epsilon\psi_{min} }{2(2Lip_T)^\alpha}$. Then  when $L_{\mc O}>\hat L$,  and $\|h\|_0<\frac{D^\a(\mc O)}{\sharp \mc O}\cdot \hat \d$, one has
\begin{align}\label{XXX-1}
\left(\frac{|a_\mathcal{O}|\|\psi\|_0+\|h\|_0}{\varepsilon}\right)^{\frac{1}{\alpha}}<\frac{D(\mathcal{O})}{2Lip_T}.
\end{align}

Firstly, we show that $Area_1$ contains all $x\in {X}$ with $G(x)\le0$.
	
	Given $x\notin Area_1$ when $Area_1\neq X$, we are to show that $G(x)>0$. There exists $y\in \mathcal{O}$ such that
	$$d(x,y)=d(x,\mathcal{O})> \left(\frac{|a_\mathcal{O}|\|\psi\|_0+\|h\|_0}{\varepsilon}\right)^{\frac{1}{\alpha}}.$$
Note that
	\begin{align}\label{X-2}
	\begin{split}
	\bar u+h-a_\mathcal{O}\psi_K\geq h-|a_\mathcal{O}|\psi_K \ge-|a_\mathcal{O}|\|\psi\|_0-\|h\|_0
	\end{split}
	\end{align}
 since $\bar u\ge0$ and $\|\psi_K\|_0\le \|\psi\|_0$. Then
	\begin{align*}G(x)&=\bar u(x)+\varepsilon d^\alpha(x,\mathcal{O})+h(x)-a_\mathcal{O}\psi_K\\
	&\ge \varepsilon d^\alpha(x,\mathcal{O})-|a_\mathcal{O}|\|\psi\|_0-\|h\|_0\\
	&> \varepsilon\cdot\left(\left(\frac{|a_\mathcal{O}|\|\psi\|_0+\|h\|_0}{\varepsilon}\right)^{\frac{1}{\alpha}}\right)^\alpha-|a_\mathcal{O}|\|\psi\|_0-\|h\|_0\\
	&=0.	
	\end{align*}

Secondly, we will show that by choosing $L_\mathcal{O}$, $\|h\|_\a$ and $\|h\|_0$ properly, for any $z\in {X}$ which is not a generic point of $\mu_\mathcal{O}$, there is an $m\in\mathbb{N}\cup\{0\}$ such that $\sum_{i=0}^{m}G(T^iz)>0$. The conditions proposed for $L_\mathcal{O}$, $\|h\|_\a$ and $\|h\|_0$ will provide the existence of the constants $\hat L$ and $\hat \d$ being requested by Proposition \ref{P:OpenDencePerMin}.

Suppose that $z\in {X}$ is not a generic point of $\mu_\mathcal{O}$. There are two cases. In the case $z\notin Area_1$ just note $m=0$ since $G(x)>0$ by claim 1.

In the case $z\in Area_1$, there is $y_0\in\mathcal{O}$ such that $$d(z,y_0)=d(z,\mathcal{O})\le\left(\frac{|a_\mathcal{O}|\|\psi\|_0+\|h\|_0}{\varepsilon}\right)^{\frac{1}{\alpha}}<\frac{D(\mathcal{O})}{2Lip_T}$$
by \eqref{XXX-1}.  If $d(T^kz,T^ky_0)\le\delta$ for all $k\in\mathbb{N}$,  by {\bf ASP},  we have
 $$d(T^kz,T^ky_0)\le Ce^{-\lambda k}(d(z,y_0)+d(T^{2k}z,T^{2k}y_0))\le 2Ce^{-\lambda k}\delta\to0\text { as }k\to+\infty.$$
 Then $z$ must be a generic point of $\mu_\mathcal{O}$ which is impossible by our assumption. Hence, there must be some $m_0\in\mathbb{N}$ such that
 $$d(T^{m_0}z,T^{m_0}y_0)>\delta>\frac{D(\mathcal{O})}{2} .$$
Let $m_1\in\mathbb{N}$ be the smallest time such that
\begin{align}\label{X-3}
\begin{split}
\frac{D(\mathcal{O})}{2Lip_T}\le d(T^{m_1}z,T^{m_1}y_0)\le\frac{D(\mathcal{O})}{2}<\delta.
\end{split}
\end{align}
 Then, we have  $$d(T^{m_1}z,\mathcal{O})=d(T^{m_1}z,T^{m_1}y_0)\ge\frac{D(\mathcal{O})}{2Lip_T}.$$
 Hence, by \eqref{X-2}, we have
 \begin{align}\label{X-5}
\begin{split}G(T^{m_1}z)&= \bar{u}(T^{m_1}z)+\varepsilon d^\alpha(T^{m_1}z,\mathcal{O})+h(T^{m_1}z)-a_\mathcal{O}\psi_K(T^{m_1}z)\\
 &\ge \varepsilon d^\alpha(T^{m_1}z,\mathcal{O})-|a_\mathcal{O}|\|\psi\|_0-\|h\|_0\\
 &\ge\varepsilon\cdot\left(\frac{D(\mathcal{O})}{2Lip_T}\right)^\alpha-|a_\mathcal{O}|\|\psi\|_0-\|h\|_0.\\
 \end{split}
\end{align}
 Let  $m_2\in\mathbb{N}$ the largest time with $0\le m_2<m_1$ such that
 $$T^{m_2}z\in Area_1.$$
 Then for all $m_2<n<m_1$ \begin{align}\label{X-12}
 G(T^nz)>0.\end{align}
 On the other hand, since $m_2<m_1$,  by (\ref{X-3}), one has that for all $0\le n\le m_2$
  $$d(T^{n}z,T^{n}y_0)\le \frac{D(\mathcal{O})}{2}<\delta.$$
  Then by {\bf ASP},  one has that for all $0\le n\le m_2$
  \begin{align*}
  d^\alpha(T^{n}z,T^{n}y_0)&\le C^\alpha e^{-\lambda\alpha \min(n,m_2-n)}\cdot(d(z,y_0)+d(T^{m_2}(z),T^{m_2}y_0))^\alpha\\
  &\le  2C^{\alpha}(e^{-\lambda\alpha n}+e^{-\lambda\alpha (m_2-n)})\frac{|a_\mathcal{O}|\|\psi\|_0+\|h\|_0}{\varepsilon}
\end{align*}
where we used the assumption that  $z,T^{m_2}z\in Area_1$.
 Therefore,
 \begin{align*}
\begin{split}\sum_{n=0}^{ m_2}d^\alpha(T^{n}z,T^{n}y_0)\le\frac{4C^{\alpha}}{1-e^{-\lambda\alpha}}\cdot\frac{|a_\mathcal{O}|\|\psi\|_0+\|h\|_0}{\varepsilon}.
 \end{split}
\end{align*}
Thus, one has that
 \begin{align}\label{X-6}
\begin{split}&\sum_{n=0}^{m_2}\left(G(T^nz)-G(T^ny_0)\right)\\
  =&\sum_{n=0}^{m_2}\left(\bar{u}(T^nz)+\varepsilon d^\alpha(T^nz,\mathcal{O})+h(T^nz)-\bar{u}(T^ny_0)-\varepsilon d^\alpha(T^ny_0,\mathcal{O})-h(T^ny_0)\right)\\
  &+a_\mathcal{O}\sum_{n=0}^{m_2}\left(\psi_K(T^ny_0)-\psi_K(T^nz)\right) \\
  \ge&\sum_{n=0}^{m_2}\left(\bar{u}(T^nz)-\bar{u}(T^ny_0)+h(T^nz)-h(T^ny_0)+a_\mathcal{O}(\psi_K(T^ny_0)-\psi_K(T^nz))\right)\\
  \ge& -(\|\bar u\|_\alpha+\|h\|_\alpha+|a_\mathcal{O}|\|\psi_K\|_\alpha)\sum_{n=0}^{m_2}d^\alpha(T^nz,T^ny_0)\\
  \ge& - (\|\bar u\|_\alpha+\|h\|_\alpha+|a_\mathcal{O}|\|\psi_K\|_\alpha) \cdot \frac{4C^{\alpha}}{1-e^{-\lambda\alpha}}\cdot\frac{|a_\mathcal{O}|\|\psi\|_0+\|h\|_0}{\varepsilon},
 \end{split}
\end{align}
where we used the fact $d^\alpha(\cdot,\mathcal{O})\ge0$ and $d^\alpha(T^ny_0,\mathcal{O})=0$.
Also note that
	\begin{align*}
	|a_\mathcal{O}|
	&=\left|\frac{\sum_{y\in\mathcal{O}}\left(\bar u(y)+h(y)\right)}{\sum_{y\in\mathcal{O}}\psi_K(y)}\right|\le \frac{ \|\bar u\|_0 + \|h\|_0 }{\psi_{min} }.	
   \end{align*}
Thus, one has that
 \begin{align}\label{XX-6}
\begin{split}&\sum_{n=0}^{m_2}\left(G(T^nz)-G(T^ny_0)\right)\\
  \ge& - (\|\bar u\|_\alpha+\|h\|_\alpha+\frac{ \|\bar u\|_0 + \|h\|_0 }{\psi_{min} }\|\psi_K\|_\alpha) \cdot \frac{4C^{\alpha}}{1-e^{-\lambda\alpha}}\cdot\frac{|a_\mathcal{O}|\|\psi\|_0+\|h\|_0}{\varepsilon}
 \end{split}
\end{align}
by \eqref{X-6}.

Note that $m_2+1=p\sharp\mathcal{O}+r$ for some nonnegative integer $p$ and $0\le r\le \sharp\mathcal{O}-1$, then  by \eqref{X-2}, one has that
 \begin{align}\label{X-7}
\begin{split}
  \sum_{n=0}^{m_2}G(T^ny_0)=\sum_{n=p\sharp\mathcal{O}}^{p\sharp\mathcal{O}+r-1}G(T^ny_0)\ge -\sharp\mathcal{O}\cdot (|a_\mathcal{O}|\|\psi\|_0+\|h\|_0)
 \end{split}
\end{align}
where we used $\int Gd\mu_\mathcal{O}=0$.

		Now we are ready to estimate $ \sum_{n=0}^{m_1}G(f^nz)$ as the following:
  \begin{align*}
 \hskip0.5cm &\sum_{n=0}^{m_1}G(T^nz)\\
 \ge& \sum_{n=0}^{m_2}G(T^nz)+G(T^{m_1}z)\\
 = &\left(\sum_{n=0}^{m_2}\left(G(T^nz)-G(T^ny_0)\right)\right)+\left(\sum_{n=0}^{m_2}G(T^ny_0)\right)+G(T^{m_1}z)\\
 \ge&-\left(\|\bar u\|_\alpha+\|h\|_\alpha+\frac{ \|\bar u\|_0 + \|h\|_0 }{\psi_{min} }\|\psi_K\|_\alpha\right) \cdot \frac{4C^{\alpha}}{1-e^{-\lambda\alpha}}\cdot\frac{|a_\mathcal{O}|\|\psi\|_0+\|h\|_0}{\varepsilon} &(\text{by }(\ref{X-6}))\\
 &-\sharp\mathcal{O}\cdot (|a_\mathcal{O}|\|\psi\|_0+\|h\|_0)&(\text{by }(\ref{X-7}))\\
  &+\varepsilon\cdot\left(\frac{D(\mathcal{O})}{2Lip_T}\right)^\alpha-|a_\mathcal{O}|\|\psi\|_0-\|h\|_0&(\text{by }(\ref{X-5}))\\
 \ge& L_1D^\a(\mc O)-L_2d_{\a, Z_{u,\psi}}(\mc O)-L_3\sharp\mathcal{O}\|h\|_0\\
 =&L_1 D^\a(\mc O)\Big(1-\frac{L_2}{L_1}\frac{1}{L_{\mathcal{O}}}-\frac{L_3\sharp\mathcal{O}}{L_1 D^\a(\mc O)}\|h\|_0\Big) ,
  \end{align*}
 where we take $\|h\|_\a\le 10\ve$ and $\|h\|_0\le \delta^\alpha$, and let
 \begin{align*}
 L_1&=\frac{\ve}{(2Lip_T)^\a},\\
 L_2&=\left(\frac{4C^{\alpha}(\|\bar u\|_\alpha+10\varepsilon+\frac{ \|\bar u\|_0 + \delta^\alpha }{\psi_{min} }\|\psi_K\|_\alpha) }{(1-e^{-\lambda\alpha})\psi_{min}\ve}+\frac{2\|\psi\|_0}{\psi_{min}}\right) \|\bar u\|_\alpha,\\
 L_3&=\left(\frac{4C^{\alpha}(\|\bar u\|_\alpha+10\varepsilon+\frac{ \|\bar u\|_0 + \delta^\alpha }{\psi_{min} }\|\psi_K\|_\alpha) }{(1-e^{-\lambda\alpha})\psi_{min}\ve}+\frac{2\|\psi\|_0}{\psi_{min}}\right)(1+\psi_{min}).
 \end{align*}
 Note that $L_1,L_2, L_3$ are positive and  depending on $\e, u, \psi$ and system constants only. By taking
 $$L_{\mc O}>2\frac{L_2}{L_1},\
  \|h\|_0<\frac{1}{2}\frac{L_1 D^\a(\mc O)}{L_3\sharp\mathcal{O}}, \text{ and } m=m_1,$$
one has that $\sum_{i=0}^mG(T^i z)>0$ provided that $z$ is not a generic point of $\mu_{\mc O}$.  Therefore, one possible choice for $\hat L$  and $\hat\d$ is to let $$\hat L=\max\left\{\frac{3L_2}{L_1},\frac{2(2Lip_T)^\alpha \|\bar u\|_\alpha}{\epsilon \psi_{min}} \right\}\text{ and } \hat\d=\min\left\{1,\frac{L_1}{3L_3},\frac{\epsilon\psi_{min} }{2(2Lip_T)^\alpha}\right\}.$$

Finally, we finish the proof by showing that when $L_{\mc O}>\hat L$, $\|h\|_{\a}\le 10\ve$, and $\|h\|_0\le \frac{D^\a(\mc O)}{\sharp \mc O}\cdot\hat\d$ with $\hat L$ and $\hat \d$ given above, the following holds
$$\int Gd\mu\ge0\text{ for all }\mu\in\mathcal{M}^e({X}, T).$$
Given a ergodic probability measure $\mu\in\mathcal{M}^e({X},T)$, in the case $\mu=\mu_{\mathcal{O}}$, we have $\int Gd\mu_{\mathcal{O}}=0$.  In the case $\mu\neq\mu_{\mathcal{O}}$, let $z$ be a generic point of $\mu$.
Note that  $z$ is not a generic point of $\mu_{\mathcal{O}}$. Thus there exists $m_1\in\mathbb{N}$ such that $$\sum_{n=0}^{m_1}G(T^nz)>0.$$
Note that  $T^{m_1+1}z$ is also not a generic point of $\mu_{\mathcal{O}}$. Thus we have $m_1+1\le m_2\in\mathbb{N}$ such that
$$\sum_{n=m_1+1}^{m_2}G(T^nz)>0.$$
By repeating the above process, we have $0\le m_1<m_2<m_3<\cdots$ such that
 $$\sum_{n=m_i+1}^{m_{i+1}}G(T^nz)>0\text{ for } i=0,1,2,3,\cdots,$$
 where $m_0=-1$. Therefore
\begin{align*}\int Gd\mu&=\lim_{i\to+\infty}\frac{1}{m_i+1}\sum_{n=0}^{m_i}G(T^nz)\\
&= \lim_{i\to+\infty}\frac{1}{m_i+1}\left(\sum_{n=0}^{m_1}G(T^nz)+\sum_{n=m_1+1}^{m_2}G(T^nz)+\cdots+\sum_{n=m_{i-1}}^{m_i}G(T^nz)\right)\\
&\ge 0.
\end{align*}
Hence, $\mu_{\mathcal{O}}\in\mathcal{M}_{min}(u+\varepsilon d^\alpha(\cdot,\mathcal{O})+h;\psi,{X},T)$.

In the end, (\ref{E:PositivityIntG}) holds ($\ve$ may need to be modified a bit) by noting that for any $\ve'>\ve_0$, the function
$G':=G+(\ve'-\ve)d(\cdot,\mc O)$ satisfies that
$$\int (G'-G)d\mu=\int (\ve'-\ve)d(\cdot,\mc O)d\mu>0\quad \forall \mu\in \mc M^e(X,T)\setminus \{\mu_{\mc O}\}.$$
This ends the proof.
\end{proof}
So far, we have accomplished the proof of Part I) of Theorem \ref{T:MainResult}.

\subsection{Proof of Part II) of Theorem \ref{T:MainResult}}\label{S:Part II)}
We will prove the following technical proposition which together with Proposition \ref{P:GoodPer}  imply the Part II) of Theorem \ref{T:MainResult}.
\begin{prop}\label{prop-3}
Let $(M,f)$ be a dynamical system on a smooth compact manifold $M$. Assume that $(X,T)$ is a subsystem of $(M,f)$, which satisfies {\bf ASP} and {\bf MCGBP}, and $T:X\rightarrow X$ is Lipschitz continuous. Then for  $0<\varepsilon<1$, $u\in \mathcal{C}^{1,0}({M})$ and strictly positive $\psi\in\mathcal{C}^{0,1}({M})$, there exist positive numbers $\hat L_1,\hat \d_1>0$ depending on $\ve,\psi,u$ and system constants only, and $\hat \d_1'>0$ depending on $\psi,u$ and system constants only (independent on $\ve$) such that the following hold: if a periodic orbit $\mathcal{O}$ of $(X,T)$ meets the following comparison condition
	\begin{align}\label{condition-2}
	D(\mathcal{O})>
\hat L_1 d_{1,Z_{u,\psi, T}}(\mathcal{O}),
	\end{align}
	then there is a $w\in \mathcal{C}^\infty(M)$ with $$\|w\|_0<\hat \d_1'\ve\text{ and }\|D_xw\|_0<2\varepsilon$$
	such that the probability measure
	$$\left\{\mu_{\mathcal{O}}:=\frac{1}{\sharp\mathcal{O}}\sum_{x\in\mathcal{O}}\delta_{x}\right\}=\mathcal{M}_{min}\left((u+w+h)|_X;\psi|_X,{X},T\right),$$
	whenever $h\in\mathcal{C}^{1,0}(M)$ satisfies $\|D_xh\|_0<5\varepsilon$ and
	$\|h\|_0<\frac{D(\mc O)}{\sharp \mc O}\cdot \hat \d_1$.
\end{prop}
Here $D_x$ is the derivative of a given function and  $Z_{u,\psi, T}$ is same as the $Z_{u,\psi}$ given by (\ref{E:Z_u,psi}) with respect to system $(X,T)$. The reason of adding subindex "$T$" is to avoid confusion on notions with such invariant set with respect to system $(M,f)$.


\begin{proof}
The proof is based on Proposition \ref{P:OpenDencePerMin} and the following approximation theorem due to Greene and Wu \cite{GW79}.
\begin{thm}\label{thm-G} Let $M$ be a smooth compact manifold. Then $\mathcal{C}^\infty(M)\cap \mathcal{C}^{0,1}(M)$ is Lip-dense in $\mathcal{C}^{0,1}(M)$.
\end{thm}
In this Theorem, {\it $\mathcal{C}^\infty(M)\cap \mathcal{C}^{0,1}(M)$ is Lip-dense in $\mathcal{C}^{0,1}(M)$} means that for any $g_1\in \mathcal{C}^{0,1}(M)$ and $\varepsilon>0$ there is a $g_2\in \mathcal{C}^\infty$ such that $\|g_1-g_2\|_0<\varepsilon$ and $\|g_2\|_1<\varepsilon+\|g_1\|_1$. Especially, $\|D_x g_2\|_0< \varepsilon+\|g_1\|_1$.

Fix $\varepsilon,\mathcal{O},\psi,u$ as in the Proposition \ref{prop-3}. Note that Proposition \ref{P:OpenDencePerMin} in the case of $\a=1$ is applicable for the current setting. Thus, by taking $\hat L_1=\hat L$ and $\mc O$ satisfying (\ref{E:GapCond}) for $\a=1$,  one has that for any $g\in \mc C^{0,1}(M)$ satisfying that $\|g\|_1\le 10\ve$ and $\|g\|_0< \frac{D(\mc O)}{\sharp \mc O}\hat \d$
$$\left\{\mu_{\mathcal{O}}:=\frac{1}{\sharp\mathcal{O}}\sum_{x\in\mathcal{O}}\delta_{x}\right\}=\mathcal{M}_{min}\left( (u+\ve d(\cdot, \mc O)+g)|_{X};\psi|_X,{X},T\right),$$
where $\hat L$, $\hat \d$ and $\mc O$ are as in Proposition \ref{P:OpenDencePerMin}. Denote that
$$\mc U_{\mc C^{0,1}}(\ve d(\cdot,\mc O)):=\left\{u+\ve d(\cdot, \mc O)+g\Big|\ g\in \mc C^{0,1}(M),\ \|g\|_1\le 10\ve,\ \|g\|_0< \frac{D(\mc O)}{\sharp \mc O}\hat \d\right\}.$$

Note that the only obstacle prevent one to derive Proposition \ref{prop-3} from Proposition \ref{P:OpenDencePerMin} directly is that $d(\cdot,\mc O)$ is only Lipschitz rather than $\mc C^1$.  A nature idea to overcome this is to find a $w\in \mc C^{1,0} (M)$ close to $\ve d(\cdot,\mc O)$ in $\mc C^{0,1}(M)$ such that an open neighborhood $\mc U_{\mc C^{1,0}}(w)$ of $u+w$ in $\mc C^{1,0}(M)$ is a subset of $\mc U_{\mc C^{0,1}}(\ve d(\cdot,\mc O))$, which is doable by applying Theorem \ref{thm-G}.

Precisely, for any $\ve_1>0$, by Theorem \ref{thm-G}, there exists a function $w\in \mc C^\infty(M)$ such that

	$$\| w\|_1\le \|\varepsilon d(\cdot,\mathcal{O})\|_1+\varepsilon_1$$
	and
	$$\| w-\varepsilon d(\cdot,\mathcal{O})\|_0<\ve_1.$$
Therefore,
\begin{align}\label{C1-1}
\|D_xw\|_0\le \varepsilon+\ve_1\text{ and }\|w\|_0\le \|\varepsilon d(\cdot,\mathcal{O})\|_0+\ve_1.
\end{align}

	Next, we choose proper $\ve_1$, $\hat \d_1$ and $\hat \d_1'$ to meet the requirement of the proposition as follows. For $h\in\mathcal{C}^{1,0}(M)$ we rewrite $u+w+h$ as $u+\varepsilon d(\cdot,\mathcal{O})+(w-\varepsilon d(\cdot,\mathcal{O})+h)$
	It remains to make $w-\varepsilon d(\cdot,\mathcal{O})+h$ satisfying the conditions of $h$ as in Proposition \ref{P:OpenDencePerMin} by adjusting $\ve_1$.
Note that
\begin{equation}\label{C1-2}
\|w-\varepsilon d(\cdot,\mathcal{O})+h\|_1\le\|w\|_1+\|\varepsilon d(\cdot,\mathcal{O})\|_1+\|h\|_1\le 2\varepsilon+\ve_1+\|h\|_1,
\end{equation}
	and
	\begin{equation}\label{C1-3}
	\|w-\varepsilon d(\cdot,\mathcal{O})+h\|_0\le\|w-\varepsilon d(\cdot,\mathcal{O})\|_0+\|h\|_0< \ve_1+\|h\|_0.
	\end{equation}
	Take
	$$\ve_1=\min\left\{\ve,\frac{D(\mc O)}{2\sharp \mc O}\cdot \hat \d\right\},\ \hat \d_1'=diam(M)+1,\ \hat\d_1=\frac12\hat\d,$$
	and let $\|h\|_1< 5\ve$ together with $\|h\|_0< \frac{D(\mc O)}{\sharp \mc O}\cdot \hat \d_1$.
	Then one has that
	\begin{align*}
	&\|D_xw\|_0\le 2\ve\text{ and }\|w\|_0\le \hat \d_1'\ve&\text{by }(\ref{C1-1})\\
	&\|w-\varepsilon d(\cdot,\mathcal{O})+h\|_1\le 8\ve<10\ve&\text{by }(\ref{C1-2})\\
	&\|w-\varepsilon d(\cdot,\mathcal{O})+h\|_0< \frac{D(\mc O)}{\sharp \mc O}\cdot \hat \d&\text{by }(\ref{C1-3}).
	\end{align*}
	Denote that
	$$\mc U_{\mc C^{1,0}}(w):=\left\{u+w+h\Big|\ h\in \mc C^{1,0}(M),\ \|h\|_1< 5\ve, \|h\|_0<  \frac{D(\mc O)}{\sharp \mc O}\cdot \hat \d_1\right\}.$$
	Thus, $\mc U_{\mc C^{1,0}}(w)\subset \mc U_{\mc C^{0,1}}(\ve d(\cdot,\mc O)) $, which is simultaneously a non-empty open subset in $\mc C^{1,0}(M)$. This complete the proof.
\end{proof}

\section{Discussions on the case of $\mathcal{C}^{s,\alpha}$ observables}\label{S:CsCase}
In this section, we consider the case when the observable functions has higher regularity. Unlike the case of  $\mc C^{0,\a}$ and $\mc C^{1,0}$ observables, only partial results are presented in this paper. To avoid unnecessarily tedious discussions, we will consider the following model which is relatively simple and illustrative.

Let $(M,f)$ be a dynamical system on a smooth compact manifold $M$ and $\psi:M\rightarrow \mathbb{R}^+$ be a strictly positive continuous function. Denote that $Per_{s,\alpha}({M},\psi,f)$ is the collection of function $u\in\mathcal{C}^{s,\alpha}(M)$ such that $\mathcal{M}_{min}(u;\psi,{M},f)$ contains only one probability measure which is periodic. Now we define $Per^*_{s,\alpha}({M},\psi,f)$ the collection of function  $u\in\mathcal{C}^{s,\alpha}(M)$ such that $\mathcal{M}_{min}(u;\psi,{M},f)$ contains at least one periodic probability measure. And $Loc_{s,\alpha}({M},\psi,f)$ is defined by
\begin{align*}Loc_{s,\alpha}({M},\psi,f)=\{&u\in Per_{s,\alpha}({M},\psi,f):\text{ there is } \varepsilon>0\text{ such that }\\
&\mathcal{M}_{min}(u+h;\psi,{M},f)=\mathcal{M}_{min}(u;\psi,{M},f)\text{ for all }\|h\|_{s,\alpha}<\varepsilon \}.\end{align*}
In the case $s\ge1$ and $\alpha>0$ or $s\ge 2$, we do not have result like Theorem \ref{T:MainResult}. But, we have the following weak version.
\begin{prop}\label{prop-4}
	Let $f:M\rightarrow M$ be a Lipschitz continuous selfmap on a smooth compact manifold $M$ and  $(M,f)$ has {\bf ASP} and {\bf MCGBP}. Let  $\psi\in \mathcal{C}^{0,1}(M)$ be strictly positive.  If $u\in\mathcal{C}(M)$ with $u\ge 0$ and there is periodic orbit $\mathcal{O}$ of $(M,f)$ such that $u|_\mathcal{O}=0,$
	then for all $\varepsilon>0$, $s\in\mathbb{N}$ and $0\le \alpha\le 1$, there is a function $w\in\mathcal{C}^\infty(M)$ with $\|w\|_{s,\alpha}<\varepsilon$ and a constant $\varrho>0$ such that the probability measure
	$$\mu_{\mathcal{O}}=\frac{1}{\sharp\mathcal{O}}\sum_{x\in\mathcal{O}}\delta_{x}\in\mathcal{M}_{min}(u+w+h;\psi,{M},f),$$
	whenever $h\in\mathcal{C}^{0,1}({M})$ with $\|h\|_1<\varrho$ and
	$\|h\|_0<\varrho.$
\end{prop}
By using proposition \ref{prop-4}, we have the following result immediately.
\begin{thm} Let $f:M\rightarrow M$ be a Lipschitz continuous selfmap on a smooth compact manifold $M$. If $(M,f)$ has {\bf ASP} and {\bf MCGBP}, then $Loc_{s,\alpha}({M},\psi)$ is an open dense subset of $Per^*_{s,\alpha}({M},\psi)$ w.r.t. $\|\cdot\|_{s,\alpha}$ for integer $s\ge 1$, real number $0\le \alpha\le 1$ and $\psi\in \mathcal{C}^{0,1}(M)$ is a strictly positive continuous function.
\end{thm}
\begin{proof}Immediately from Remark \ref{rem-reveal} and Proposition \ref{prop-4} .
	\end{proof}
Without result like Proposition \ref{P:OpenDencePerMin}, we can not get the full result about the generality of $\mathcal{C}^{s,\alpha}(M), s\ge 2$ or $s\ge 1$ and $\alpha >0$. So we rise the following question:
\begin{ques}By a expanding map $f:\mathbb{T}\to\mathbb{T}:x\to 2x$, is there a  $u\in\mathcal{C}^{s,\alpha}(\mathbb{T})$, $s\ge1$ and $0<\alpha\le 1$ or $s\ge 2$ such that any function near $u$ w.r.t. $\|\cdot\|_{s,\alpha}$ has no periodic minimizing measure?
\end{ques}	
At last, we complete the proof of proposition \ref{prop-4}.
\begin{proof}[Proof of proposition \ref{prop-4}] Fix $\varepsilon, s,\alpha,\mathcal{O},\psi$ as in proposition. $C$ and $\delta$ are the constants as in {\bf ASP} and $Lip_f$ is definited as in \eqref{lip-definition}.	
	Just take $w\in\mathcal{C}^\infty$ with $\|w\|_{s,\alpha}<\varepsilon$, $w|_\mathcal{O}=0$ and $w|_{M\setminus\mathcal{O}}>0$. For $0\le r\le D(\mathcal{O})$, we note 
 $$\theta(r)=\min\{w(x):d(x,\mathcal{O})\ge  r, x\in{M}\}.$$
It is clear that  $\theta(0)=0$, $\theta(r)>0$ for $ r\neq 0$ and $\theta$ is non-decreasing. Now we fix the constants
	$$0<\rho<\frac{D(\mathcal{O})}{2Lip_f}$$
	and positive $\varrho$ smaller than
$$
\begin{small}
\min\left\{\frac{\theta(\rho)\psi_{min}}{\psi_{min}+\|\psi\|_0},\theta\left(\frac{D(\mathcal{O})}{2Lip_f}\right)
\cdot\frac{1-e^{-\lambda}}{4C\rho},\frac{1}{2}\cdot\frac{\theta\left(\frac{D(\mathcal{O})}{2Lip_f}\right)}{\frac{2C\rho\|\psi\|_1}{(1-e^{-\lambda})\psi_{min}}+\sharp\mathcal{O}
+\sharp\mathcal{O}\frac{\|\psi\|_0}{\psi_{min}}+\frac{\|\psi\|_0}{\psi_{min}}+1}\right\}.
\end{small}$$
 By fixing $h\in \mathcal{C}^{0,1}({M})$ with $\|h\|_1<\varrho$ and $\|h\|_0<\varrho$, we are to show that $\mu_{\mathcal{O}}\in\mathcal{M}_{min}(w+h;\psi,{M},f)$ which implies that $\mu_{\mathcal{O}}\in\mathcal{M}_{min}( u+w+h;\psi,{M},f)$ since $u\ge 0$ and $u|_\mathcal{O}=0$ by assumption.\\
	Note that $G=w+h-a_\mathcal{O}\psi$,  where $a_{\mathcal{O}}:=\frac{\sum_{y\in\mc O}(w+h)(y)}{\sum_{y\in\mc O}\psi(y)}$. It is straightforward to see that
	\begin{align}\label{b-1}
	\begin{split}
	|a_{\mathcal{O}}|&\le\frac{\|h\|_0}{\psi_{min}},
	\end{split}
	\end{align}
	where we used $w|_\mathcal{O}=0$. Then $\frac{\int Gd\mu}{\int \psi d\mu}=\frac{\int w+hd\mu}{\int \psi d\mu}-a_\mathcal{O}.$
	Therefore, to show that $\mu_{\mathcal{O}}\in\mathcal{M}_{min}(w+h;\psi,{M},f)$, it is enough to show that
	$$\int Gd\mu\ge0\text{ for all }\mu\in\mathcal{M}^e({M},f),$$
	where we used the assumption $\psi$ is strictly positive and the fact $\int Gd\mu_\mathcal{O}=0$.\\
	{\it {\bf Claim 1.} Put $ Area_1=\{y\in{M}:d(y,\mathcal{O})\le \rho\}$, then $Area_1$ contains all $x\in {M}$ with $G(x)\le0$. }
	\begin{proof}[Proof of claim 1.]
		For $x\notin Area_1$, we have
		\begin{align*}
		G(x)&= (w+h-a_\mathcal{O}\psi)(x)\ge \theta(\rho)-|a_\mathcal{O}|\|\psi\|_0-\|h\|_0\\
		&\ge \theta(\rho)-\frac{\|\psi\|_0+\psi_{min}}{\psi_{min}}\|h\|_0\\
		&>\theta(\rho)-\frac{\|\psi\|_0+\psi_{min}}{\psi_{min}}\varrho\\
		&\ge0.	
		\end{align*}
	This ends the proof of claim 1.\end{proof}		
	{\it {\bf Claim 2.} If $z\in {M}$ is not a generic point of $\mu_\mathcal{O}$, then there is $m\in\mathbb{N}\cup\{0\}$ such that $\sum_{i=0}^mG(f^iz)>0$.}
	\begin{proof}[Proof of claim 2.] 	If $z\notin Area_1$, just note $m=0$,
		we have nothing to prove.
		
		Now we assume that $z\in Area_1$. There is $y_0\in\mathcal{O}$ such that $$d(z,y_0)=d(z,\mathcal{O})\le\rho <\frac{D(\mathcal{O})}{2Lip_f}<\delta.$$
		If  $d(f^kz,f^ky_0)\le \delta$ for  all $k\ge 0$,  by {\bf ASP},  we have
		$$d(f^kz,f^ky_0)\le Ce^{-\lambda k}(d(z,y_0)+d(f^{2k}z,f^{2k}y_0))\le 2Ce^{-\lambda k}\delta\to0\text { as }k\to+\infty.$$ Hence, $z$ is a generic point of $\mu_\mathcal{O}$ which is impossible by our assumption.
	Therefore, there must be some $m_1>0$ such that $d(f^{m_1} z,f^{m_1}y_0)\ge\delta.$
		There exists $m_2>0$ the smallest time such that
		\begin{align}\label{b-2}
		\frac{D(\mathcal{O})}{2Lip_f}\le d(f^{m_2} z,f^{m_2}y_0)\le\frac{D(\mathcal{O})}{2},
		\end{align}	
		where we used the assumption $f$ is Lipschitz with Lipschitz constant $Lip_f$.
		Then we have $d(f^{m_2} z,\mathcal{O})=d(f^{m_2} z,f^{m_2}y_0)\ge 	\frac{D(\mathcal{O})}{2Lip_f}$ and
		\begin{align}\label{b-3}
		\begin{split}G(f^{m_2} z)=(w+h-a_\mathcal{O}\psi)(f^{m_2} z)&\ge \theta\left(\frac{D(\mathcal{O})}{2Lip_f}\right)-\|h\|_0- |a_\mathcal{O}|\|\psi\|_0.\\
		\end{split}
		\end{align}
		where we used the definition of $\theta(\cdot)$.
		On the other hand,
		$\frac{D(\mathcal{O})}{2Lip_f}>\rho$ by assumption which implies that
		\begin{align}\label{b-4}
		f^{m_2}z\notin Area_1.
		\end{align}
		We take  $m_3$ the largest time with $0\le m_3\le m_2$ such that
		$$f^{m_3}z\in Area_1,$$
		where we use the assumption $z\in Area_1$.
		By \eqref{b-4}, it is clear that $m_3<m_2$ since $m_2$ is the smallest time meets \eqref{b-2}.
		Then by claim 1, \begin{align}\label{b-5}
		G(f^{n}z)>0\text{  for all }m_3<n<m_2.\end{align}
		Additionally, by the choice of $m_2$ and \eqref{b-2} one has that
		$$d(f^nz,f^ny_0)\le \frac{D(\mathcal{O})}{2}\le \delta\text{ for all } 0\le n\le m_3.$$ Therefore, by {\bf ASP}, we have for all $0\le n\le m_3$,
		\begin{align*}
		d(f^nz,f^ny_0)\le C\rho(e^{-\lambda n}+e^{-\lambda(m_3-n)}),
		\end{align*}
		where we used $z,f^{m_3}z\in Area_1$.
		Hence,
		\begin{align*}
	\sum_{n=0}^{m_3}d(f^nz,f^ny_0)\le \frac{2C\rho}{1-e^{-\lambda}}.
		\end{align*}
		Since $w\ge 0$ and $w|_\mathcal{O}=0$, one has
		\begin{align}\label{b-6}
		\begin{split}&	\sum_{n=0}^{m_3}\left(G(f^nz)-G(f^ny_0)\right)\\
		=&\sum_{n=0}^{m_3}\left(w(f^nz)-w(f^ny_0)+h(f^nz)-h(f^ny_0)+a_\mathcal{O}\psi(f^ny_0)-a_\mathcal{O}\psi(f^nz)\right)\\
		\ge& -(\|h\|_1+|a_\mathcal{O}|\|\psi\|_1)\sum_{n=0}^{m_3}d(f^nz,f^ny_0)\\
		\ge& - (\|h\|_1+|a_\mathcal{O}|\|\psi\|_1) \cdot \frac{2C\rho}{1-e^{-\lambda}}.
		\end{split}
		\end{align}
		By assuming that $m_3=p\sharp\mathcal{O}+q$ for some nonnegative integer $p$ and $0\le q\le \sharp \mathcal{O}-1$, one has
		\begin{align}\label{b-7}
		\begin{split}
		\sum_{n=0}^{m_3}G(f^ny_0)&=\sum_{m_3-q-1}^{m_3}G(f^ny_0)\ge -\sharp\mathcal{O}\cdot (\|h\|_0+|a_\mathcal{O}|\|\psi\|_0).\\
		\end{split}
		\end{align}
		where we used the facts $\int Gd\mu_\mathcal{O}=0$ and $G\ge -\|h\|_0+|a_\mathcal{O}|\|\psi\|_0$.
		Combining \eqref{b-1},\eqref{b-3}, \eqref{b-5}, \eqref{b-6} and \eqref{b-7}, we have
		\begin{align*}
		\sum_{n=0}^{m_2}G(f^nz)\ge& \sum_{n=0}^{m_3}G(f^nz)+G(f^{m_2}z)\\
		=& \sum_{n=0}^{m_3}\left(G(f^nz)-G(f^ny_0)\right)+\sum_{n=0}^{m_3}G(f^ny_0)+G(f^{m_2}z)\\
		\ge& - (\|h\|_1+|a_\mathcal{O}|\|\psi\|_1) \cdot \frac{2C\rho}{1-e^{-\lambda}}-\sharp\mathcal{O}\cdot (\|h\|_0+|a_\mathcal{O}|\|\psi\|_0)\\
		&+\theta\left(\frac{D(\mathcal{O})}{2C}\right)-\|h\|_0- |a_\mathcal{O}|\|\psi\|_0\\
		=&\theta\left(\frac{D(\mathcal{O})}{2C}\right)-\|h\|_1  \cdot \frac{2C\rho}{1-e^{-\lambda}}\\
		&-\left( \frac{2C\rho\|\psi\|_1}{(1-e^{-\lambda})\psi_{min}}+\sharp\mathcal{O}+\sharp\mathcal{O}\frac{\|\psi\|_0}{\psi_{min}}+\frac{\|\psi\|_0}{\psi_{min}}+1\right)\|h\|_0\\
		>&0,
		\end{align*}
		where we used the assumption of $h$. Therefore, $m=m_2$ is the time we need. This ends the  proof of claim 2.
	\end{proof}
	
	Now we end the proof. It is enough to show that for all $\mu\in\mathcal{M}^e({M},f)$
	$$\int Gd\mu\ge0.$$
	Given $\mu\in\mathcal{M}^e(f)$, in the case $\mu=\mu_{\mathcal{O}}$, it is obviously true.  In the case $\mu\neq\mu_{\mathcal{O}}$, just let $z$ be a generic point of $\mu$. Note that  $z$ is not a generic point of $\mu_{\mathcal{O}}$. By claim 2, we have $m_1\in\mathbb{N}$ such that $$\sum_{n=0}^{m_1}G(f^nz)>0$$
	Note that  $f^{m_1+1}z$ is also not a generic point of $\mu_{\mathcal{O}}$. By claim 2, we have $m_2\ge m_1+1$ such that $$\sum_{n=m_1+1}^{m_2}G(f^nz)>0.$$
	By repeating the above process, we have $0\le m_1<m_2<m_3<\cdots$ such that
	$$\sum_{n=m_i+1}^{m_{i+1}}G(f^nz)>0,i=0,1,2,3,\cdots,$$
	where $m_0$ is noted by $-1$. Therefore
\begin{align*}\int Gd\mu&=\lim_{i\to+\infty}\frac{1}{m_i+1}\sum_{n=0}^{m_i}G(f^nz)\\
&= \lim_{i\to+\infty}\frac{1}{m_i+1}\left(\sum_{n=0}^{m_1}G(f^nz)+\sum_{n=m_1+1}^{m_2}G(f^nz)+\cdots+\sum_{n=m_{i-1}}^{m_i}G(f^nz)\right)\\
&\ge 0.
\end{align*}
	That is, we have $\mu_{\mathcal{O}}\in\mathcal{M}_{min}(u+h;\psi,{M},f)$ by our beginning discussion. This ends the proof.\end{proof}

\appendix

\section{{Ma\~n\'e-Conze-Guivarc'h-Bousch's Property}}\label{Sec-manelemma}

In this section, we mainly present Bousch's work (see \cite{Bousch_Mane} for detail) to show that uniformly hyperbolic  diffeomorphism on a smooth compact manifold has {\bf MCGBP}. The same argument shows that the Axiom A attractor  also has {\bf MCGBP}.


Let $f:M\rightarrow M$ be a diffeomorphism on a smooth compact manifold $M$.  By a function $u:{M}\to\mathbb{R}$ and an integer $K\ge 1$, note  $u_K=\frac{1}{K}\sum_{i=0}^{K-1}u\circ f^i$.
	Since $f$ is assumed Lipschitz,  $u\in\mathcal{C}^{0,\alpha}(M)$ for some $0<\alpha\le 1$ implies that $u_K\in\mathcal{C}^{0,\alpha}(M)$.
	Additionally, one has that
	\begin{align}\label{eq15-1}
	\int u d\mu=\int u_K d\mu \text{ for all } \mu\in\mathcal{M}({M},f).
	\end{align}
	Therefore, $$\beta(u;{M},f)=\beta(u_K;{M},f)\text{ and } \mathcal{M}_{min}(u;{M},f)=\mathcal{M}_{min}(u_K;{M},f).$$

\begin{thm}\label{thm-2.5}Let $f:M\rightarrow M$ be an Anosov diffeomorphism on a smooth compact manifold $M$. Then for $0<\alpha\le 1$, there is an integer $K=K(\alpha)$ such that for all $u\in \mathcal{C}^{0,\alpha}({M})$ with $\beta(u;{M},f)\ge 0$,  there is a function $v\in \mathcal{C}^{0,\alpha}({M})$ such that $u_K\ge v\circ f^K-v$.
\end{thm}
\begin{proof} The proof mainly follows Bousch's work in \cite{Bousch_Mane}, to which we refer readers for detailed proof. It is worth to point out that the only difference is that, in our setting, one needs a large integer $K=K(\alpha)$ to grantee that $(M, f^K)$ meets the condition in \cite{Bousch_Mane},  and  to replace $u$ by $u_K$ as $\beta(u;{M},f)=\beta(u_K;{M},f)$.
\end{proof}

\section*{Acknowledgement}
At the end, we would like to express our gratitude to Tianyuan Mathematical Center in Southwest China, Sichuan University and Southwest Jiaotong University for their support and hospitality.


\begin{thebibliography}{HSY}

\MRbibitem[BLL]{3114331}{BLL}
\textsc{Baraviera, A.; Leplaideur, R.; Lopes, A.O.}
\textit{Ergodic optimization, zero temperature limits and the max-plus algebra.}
IMPA, Rio de Janeiro, 2013.

\bibitem[B]{Bochi}
\textsc{Bochi, J.}
Ergodic optimization of Birkhoff averages and Lyapunov exponents.
\textit{Proceedings of the International Congress of Mathematicians} 2018, Rio de Janeiro, vol. 2, pp. 1821--1842.

\MRbibitem[BZ]{3567264}{BZ}
\textsc{Bochi, J.;  Zhang, Y.} Ergodic optimization of prevalent super-continuous
functions. \textit{Int. Math. Res. Not. IMRN} 19 (2016), 5988--6017.

\MRbibitem[Bo1]{Bo1}{Bousch_poisson}
\textsc{Bousch, T.}
Le poisson n'a pas d'ar\^{e}tes.
\textit{Ann.\ Inst.\ H.\ Poincar\'{e} Probab.\ Statist.\ }36 (2000), no.\ 4, 489--508.


\MRbibitem[Bo2]{Bo2}{Bousch_Walters}
\textsc{Bousch, T.}
La condition de Walters.
\textit{Ann.\ Sci.\ \'{E}cole Norm.\ Sup.\ }34 (2001), no.\ 2, 287--311.



\MRbibitem[Bo3]{Bo3}{Bousch_Mane}
\textsc{Bousch, T.}
Le lemme de Ma\~{n}\'{e}--Conze--Guivarc'h pour les syst\`{e}mes amphidynamiques rectifiables.
\textit{Ann.\ Fac.\ Sci.\ Toulouse Math.\ } 20 (2011), no.\ 1, 1--14.

\bibitem[Bo4]{Bousch_note}
\textsc{Bousch, T.}
Generity of minimizing periodic orbits after Contreras.
\textit{British Math. Colloquium talk.} QUML. April, 2014.

\MRbibitem[BQ]{2995885}{BQ}
\textsc{Bressaud, X; Quas, A.}
Rate of approximation of minimizing measures.
\textit{Nonlinearity,}  20 (2007), no. 845-853.


\MRbibitem[BF]{BF}{BF}
\textsc{Barral, J.; Feng, D.}
Weighted thermodynamic formalism on subshifts and applications.
\textit{Asian J. Math.} 16 (2012), no. 2, 319–352.

\MRbibitem[BCW]{BCW}{BCW}
\textsc{Barreira, L.;  Cao, Y.; Wang, J.}
Multifractal analysis of asymptotically additive sequences
\textit{J.\ Stat.\ Phys.} 153 (2013), no. 5, 888--910.


\MRbibitem[Co1]{3529118}{Contreras_EO}
\textsc{Contreras, G.}
Ground states are generically a periodic orbit.
\textit{Invent. math.} 205 (2016), 383--412

\bibitem[Co2]{Contreras_Mane}
\textsc{Contreras, G.}
Generic Ma\~{n}e sets.
\textit{arxiv1410.7141.}



\MRbibitem[CLT]{1855838}{CLT}
\textsc{Contreras, G.; Lopes, A.O.; Thieullen, P.}
Lyapunov minimizing measures for expanding maps of the circle.
\textit{Ergodic Theory Dynam.\ Systems} 21 (2001), no.\ 5, 1379--1409.


\MRbibitem[FH]{2645746}{FH}
\textsc{Feng, D.; Huang, W.}
Lyapunov spectrum of asymptotically sub-additive potentials.
\textit{Comm.\ Math.\ Phys.} 297 (2010), no. 1, 1--43.






\MRbibitem[Je1]{2191393}{Jenkinson_survey}
\textsc{Jenkinson, O.}
Ergodic optimization.
\textit{Discrete Contin.\ Dyn.\ Syst.\ }15 (2006), no.\ 1, 197--224.

\bibitem[Je2]{Jenkinson_newsurvey}
\textsc{Jenkinson, O.}
Ergodic optimization in dynamical systems.
\textit{Ergodic Theory Dynam.\ Systems} 2018, to appear.


\MRbibitem[Ma]{1384478}{Mane}
\textsc{Ma\~{n}\'{e}, R.}
Generic properties and problems of minimizing measures of Lagrangian systems.
\textit{Nonlinearity} 9 (1996), no.\ 2, 273--310.


\MRbibitem[Mo]{2412323}{Morris_entropy}
\textsc{Morris, I.D.}
Maximizing measures of generic H\"older functions have zero entropy.
\textit{Nonlinearity} 21 (2008), no.\ 5, 993--1000.





\bibitem[GW]{GW79}
\textsc{Greene, R.E.; Wu, H.}
$C^\infty$ approximations of convex, subharmonic and plurisubharmonic functions.
\textit{ Ann. Sci. \'Ecole Norm.\ Sup. } 12 (1979), no.\ 4,  47--84.


\MRbibitem[QS]{2995885}{QS}
\textsc{Quas, A.; Siefken, J.}
Ergodic optimization of super-continuous functions on shift spaces.
\textit{Ergodic Theory Dynam.\ Systems} 32 (2012), no.\ 6, 2071--2082.




\bibitem[SGOY]{SGOY93}
\textsc{Shinbort, T.; Grebogi, C.; Ott, E.; Yorke, J.}
Using small perturbation to control chaos.
\textit{Nature} 363 (1993),  411--7.


\MRbibitem[YH]{1709845}{YH}
\textsc{Yuan, G.; Hunt, B.R.}
Optimal orbits of hyperbolic systems.
\textit{Nonlinearity} 12 (1999), no.\ 4, 1207--1224.

\MRbibitem[OGY]{1041523}{OGY90}
\textsc{Ott, E.; Grebogi, C.; Yorke, J.}
Controlling chaos.
\textit{Phys.\ Rev.\ Lett. 64 (1990),  11--96.}




\end{thebibliography}
\end{document}